\documentclass[review]{elsarticle}

\usepackage{lineno,hyperref}
\modulolinenumbers[5]

\journal{Journal of \LaTeX\ Templates}

\usepackage[centertags]{amsmath}
\usepackage{amssymb}
\usepackage{amsthm}
\usepackage[misc]{ifsym}
\usepackage{mathrsfs}
\usepackage{graphicx}
\usepackage{epstopdf}
\usepackage{pgf,tikz}
\usetikzlibrary{arrows}
\usepackage{multicol}
\usepackage{graphicx,xcolor,subfigure}

\newcommand{\Mod}[1]{\ (\mathrm{mod}\ #1)}

\theoremstyle{plain}
\newtheorem{thm}{Theorem}[section]
\newtheorem{cor}[thm]{Corollary}
\newtheorem{lem}[thm]{Lemma}
\newtheorem{prop}[thm]{Proposition}

\newtheorem{conj}[thm]{Conjecture}

\theoremstyle{definition}
\newtheorem{defn}[thm]{Definition}

\begin{document}

\begin{frontmatter}

\title{The relation between Hamiltonian and $1$-tough properties of the Cartesian product graphs}

\author[mymainaddress]{Louis Kao\corref{mycorrespondingauthor}}
\cortext[mycorrespondingauthor]{Corresponding author}
\ead{chihpengkao.am03g@g2.nctu.edu.tw}

\author[mymainaddress]{Chih-wen Weng}

\address[mymainaddress]{Department of Applied Mathematics, National Chiao Tung University,\\ 1001 Ta Hsueh Road, Hsinchu, Taiwan.}

\begin{abstract}
The relation between Hamiltonicity and toughness of a graph is a long standing research problem.
The paper studies the Hamiltonicity of the Cartesian product graph $G_1\square G_2$ of graphs $G_1$ and $G_2$ satisfying that $G_1$ is traceable and $G_2$ is connected with a path factor.
Let $P_n$ be the path of order $n$ and $H$ be a connected bipartite graph. With certain requirements of $n$, we show that the following three statements are equivalent: (i) $P_n\square H$ is Hamiltonian; (ii) $P_n\square H$ is $1$-tough; and
(iii) $H$ has a path factor.
\end{abstract}


\end{frontmatter}

%\linenumbers

\section{Introduction}\label{s1}

A graph is  {\it Hamiltonian} if it contains a spanning cycle, and is {\it traceable} if it contains a spanning path. Determining the Hamiltonicity of a given graph is an old and famous problem.
Here we focus on a family of graphs called Cartesian product graphs.
\begin{defn}
The {\it Cartesian product graph} $G_1\square G_2$ of graphs $G_1$ and $G_2$ is a graph with vertex set
$$V(G_1\square G_2)=\{v_u~|~v\in V(G_1),u\in V(G_2)\},$$
and edge set $E(G_1\square G_2)=$
$$\{v_uv_w~|~v\in V(G_1),uw\in E(G_2)\}\cup\{v_uw_u~|~u\in V(G_2), vw\in E(G_1)\}.$$
\end{defn}

Let $\Delta(G)$ denote the maximum degree of graph $G$ and $|V(G)|$ denote the number of vertices of $G$.   The {\it order} of $G$ is $|V(G)|$. Let $P_n$ denote a path of order $n$. A {\it path factor} of a graph is a spanning subgraph of the graph such that each component of the spanning subgraph is isomorphic to a path with order at least two. If each component in a path factor is isomorphic to $P_2$, the path factor is called a {\it perfect matching}.
We will prove the following theorem.

\begin{thm}\label{main} Let $G_1$ be a traceable graph and $G_2$ a connected graph. Statements (a) and (b) are given as following:
\begin{enumerate}
\item[(a)] $G_2$ has a perfect matching and $|V(G_1)|\geq \Delta(G_2)$.
\item[(b)] $G_2$ has a path factor and $|V(G_1)|$ is an even integer with $|V(G_1)|\geq 4\Delta(G_2)-2$.
\end{enumerate}
If one of (a),(b) holds, then $G_1\square G_2$ has a Hamiltonian cycle.
\end{thm}

The well-known Petersen's matching theorem \cite{p:91} states that a connected $3$-regular graph with no cut-edges has a perfect matching, so together with Theorem~\ref{main}(a) we obtain the following corollary.

\begin{cor}\label{c1}  Let $G_1$ be a traceable graph of order at least $3$. If $G_2$ is a connected $3$-regular graph with no cut-edge, then $G_1\square G_2$ has a Hamiltonian cycle.\qed
\end{cor}

Let $\delta(G)$  denote the minimum degree of graph $G$.  We use  Theorem~{\ref{main}}(b) to obtain the following two Dirac-type results \cite{d52}.

\begin{cor}\label{c2}  Let $G_2$ be a connected graph with $2\delta(G_2)\geq \Delta(G_2)$ and $G_1$ be a traceable graph of even order. If $|V(G_1)|\geq 4\Delta(G_2)-2$, then $G_1\square G_2$ has a Hamiltonian cycle.
\end{cor}

\begin{cor}\label{thm0}
Let $G_2$ be a connected graph with $\delta(G_2)\geq |V(G_2)|/3$ and  $G_1$ be a traceable graph of even order. If $|V(G_1)|\geq 4\Delta(G_2)-2$, then $G_1\square G_2$ has a Hamiltonian cycle.
\end{cor}

For $S\subseteq V(G)$ let $G-S$ denote the subgraph of $G$ induced on $V(G)-S$. To discuss the Hamiltonicity of graphs, another measure of graphs is usually considered. A graph $G$ is {\it $t$-tough} if $t$ is a rational number such that  $|S|\geq t\cdot c(G-S)$ for
any {\it cut set} $S$ of $G$,  i.e. $S\subseteq V(G)$ such that $G-S$ has $c(G-S)$ components with $c(G-S)\geq 2$. If G is not complete, the largest $t$ makes $G$ to be $t$-tough is called the {\it toughness} of $G$, denoted by
$t(G)$. For convenience, we set $t(K_n) =+\infty$, where $K_n$ is the complete graph of order $n$.

Toughness is a non-decreasing (with respect to the number of edges) graph property. Therefore, a Hamiltonian graph is $1$-tough since it contains a spanning cycle which is $1$-tough. However, not all $1$-tough graphs are Hamiltonian. Figure~1 gives a $1$-tough non-Hamiltonian graph of order $7$.

\bigskip
\begin{center}
\begin{tikzpicture}[line cap=round,line join=round,>=triangle 45,x=0.5cm,y=0.5cm]

\clip(-15.5,5.5) rectangle (-8.5,14.5);
\draw [line width=2.pt] (-15.,12.)-- (-9.,12.);
\draw [line width=2.pt] (-12.,14.)-- (-15.,12.);
\draw [line width=2.pt] (-12.,14.)-- (-9.,12.);
\draw [line width=2.pt] (-9.,12.)-- (-9.,8.);
\draw [line width=2.pt] (-9.,8.)-- (-12.,6.);
\draw [line width=2.pt] (-12.,6.)-- (-15.,8.);
\draw [line width=2.pt] (-15.,12.)-- (-15.,8.);
\draw [line width=2.pt] (-12.,14.)-- (-12.,10.);
\draw [line width=2.pt] (-12.,10.)-- (-12.,6.);
\begin{scriptsize}
\draw [fill=white] (-12.,14.) circle (3.5pt);
\draw [fill=white] (-15.,12.) circle (3.5pt);
\draw [fill=white] (-15.,8.) circle (3.5pt);
\draw [fill=white] (-12.,6.) circle (3.5pt);
\draw [fill=white] (-9.,12.) circle (3.5pt);
\draw [fill=white] (-9.,8.) circle (3.5pt);
\draw [fill=white] (-12.,10.) circle (3.5pt);
\end{scriptsize}
\end{tikzpicture}
\medskip

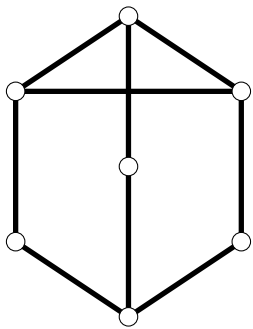
{\bf Figure 1.} A $1$-tough non-Hamiltonian graph with $7$ vertices
\end{center}
\bigskip

The idea of graph toughness was first introduced by V. Chv\'{a}tal in his 1973's seminal paper \cite{c:73}. He conjectured that there exists a real number $t_0$ such that all $t_0$-tough graphs are Hamiltonian. However, this conjecture is still open. From papers \cite{ho:95} and \cite{bbv:00}, there are examples of non-Hamiltonian graphs with toughness greater than $1.25$ and $2$, respectively. On the other hand, for specific graph classes, there may exist a toughness bound to ensure the Hamiltonicity. For instance, \cite{kk:17} shows that every $10$-tough chordal graphs are Hamiltonian.

Chv\'{a}tal's Conjecture holds  trivially  for bipartite graphs by choosing $t_0=1+\epsilon$ for any $\epsilon >0$
since a bipartite graph has toughness at most $1$.  Hence the Hamiltonicity of a $1$-tough bipartite graph deserves a further study. We apply Theorem~\ref{main} on a special family of bipartite graphs and obtain the following two corollaries.

\begin{cor}\label{thm2}
Let $T$ be a tree with a perfect matching and $n$ be a positive integer. The following three statements are equivalent:
\begin{enumerate}
\item[(1)] $P_n\square T$ is Hamiltonian.
\item[(2)] $P_n\square T$ is $1$-tough.
\item[(3)] $n\geq \Delta(T)$.
\end{enumerate}
\end{cor}

\begin{cor}\label{thm1}
Let $H$ be a connected bipartite graph. Let $n$ be an even integer and  $n\geq 4\Delta(H)-2 $. The following three statements are equivalent:
\begin{enumerate}
\item[(1)] $P_n\square H$ is Hamiltonian.
\item[(2)] $P_n\square H$ is $1$-tough.
\item[(3)] $H$ has a path factor.
\end{enumerate}
\end{cor}

The paper is organized as follows. In Section~\ref{s3}, we use a known characterization of a graph with a path factor to study the toughness of $P_n\square H$, where $H$ is a bipartite graph. We will prove Theorem~\ref{main}(a) in Section~\ref{s_4}; and prove Theorem~\ref{main}(b) in Section~\ref{s4}. In Section~\ref{s2}, we complete the proofs of  Corollary~\ref{c2}, Corollary~\ref{thm0}, Corollary~\ref{thm2} and Corollary~\ref{thm1}. To show the assumption
$|V(G_1)|\geq 4\Delta(G_2)-2$ in Theorem~\ref{main} {\it (b)} can not be extended to $|V(G_1)|\geq \Delta(G_2)$ as suggested by Corollary~\ref{thm2}, we give an example of $1$-tough non-Hamiltonian graph $P_4\square T$ for a particular tree $T$ that has a path factor and $\Delta(T)=3$  in Section~\ref{s2}. Finally, two conjectures will be given in Section~\ref{concluding}.

\section{Path factor of a bipartite graph}\label{s3}

To introduce properties of a graph with a path factor, we need more notations. First, we say a graph to have a $\{P_2,P_3\}-$factor if it has a spanning subgraph such that each component is isomorphic to $P_2$ or $P_3$. Next, we use $i(G)$ to denote the number of isolated vertices of $G$.

A $\{P_2,P_3\}$-factor is a path factor, and a path with order at least $2$ has a $\{P_2,P_3\}$-factor. Therefore, the following lemma follows.

\begin{lem}\label{p3.1}
A graph $G$ has a path factor if and only if $G$ has a $\{P_2,P_3\}$-factor. \qed
\end{lem}

The proposition below is from \cite{aae:80}.

\begin{prop}[\cite{aae:80}]\label{p3.2}
A graph $G$ has a path factor if and only if $i(G-S)\leq 2|S|$ for all $S\subseteq V(G).$ \qed
\end{prop}

\begin{lem}\label{l2.3}
Let $G$ be a graph. If $\delta(G)\geq |V(G)|/3$, then $G$ has a path factor.
\end{lem}
\begin{proof} Suppose $G$ has no path factor.  Choose $S\subseteq V(G)$ with $|I|=i(G-S)> 2|S|$ by Proposition~\ref{p3.2},
where $I$ is the set of isolated vertices in $G-S$.
As each vertex in $I$ has degree at most $|S|$ in $G$, we have
$|S|<(|S|+|I|)/3\leq |V(G)|/3$, a contradiction to the assumption that  $\delta(G)\geq |V(G)|/3$.
\end{proof}

Restricted to bipartite graphs, the following is a supplementary of Proposition~\ref{p3.2}.

\begin{prop}\label{p3.3}
If $H$ is a bipartite graph that does not contain a path factor, then there exists a vertex subset $S$ that belongs to a single partite set of $H$ with $i(H-S)>2|S|$.
\end{prop}
\begin{proof}
By Proposition~\ref{p3.2} there exists $S'\subseteq V(H)$ such that $i(H-S')>2|S'|$. Let $H$ have partite sets $A, B$ and $S_A:=S'\cap A, S_B:=S'\cap B$. Note that an isolated vertex in $H-S'$ is either an isolated vertex in $H-S_A$ or an isolated vertex in $H-S_B$. So $i(H-S_A)+i(H-S_B)=i(H-S')>2|S'|=2|S_A|+2|S_B|$ which implies $i(H-S_A)>2|S_A|$ or $i(H-S_B)>2|S_B|$.
\end{proof}

For convenience, assume
$$V(P_n)=\{1,2,\ldots,n\},E(P_n)=\{i(i+1):i=1,2,\ldots,n-1\}$$
in the rest part of this paper.

\begin{thm}\label{2.4}
If $H$ be a bipartite graph without path factors, then the Cartesian product $P_n\square H$ is not $1$-tough.
\end{thm}
\begin{proof}
By Proposition~\ref{p3.3}, there exists a vertex subset $S$ in a single partite set of $H$ such that $i(H-S)>2|S|.$
Let $I$ denote the set of isolated vertices in $H-S$ and $j_S:=\{j_s~|~s\in S\},$ $j_I:=\{j_u~|~u\in I\}$ for $1\leq j\leq n.$
Let $V(P_n\square H)=X\cup Y$ be a bipartition of $P_n\square H$ with $|X|\leq |Y|$. For the case $|X|=|Y|$, let $Y$ be the partite set which contains $1_S$. Note that $1_I, 2_S\subseteq X$, $2_I\subseteq Y$,  and $2|1_S|=2|S|<i(H-S)=|1_I|.$
If $|X|<|Y|$, then $c(P_n\square H-X)=|Y|>|X|$, implying that $P_n\square H$ is not $1$-tough.
Suppose $|X|=|Y|$. Set $X'=(X\cup 1_S)-1_I$ and $Y'=(Y\cup 1_I)-1_S.$ Now $1_I, 2_I\subseteq Y'.$
Since $1_u2_u$ is the only possible edge in $Y'$ for each $u\in I$, we have $c(P_n\square H-X')\geq |Y'|-|1_I|=|Y|-|1_S|>|X|+|1_S|-|1_I|=|X'|$. Thus $P_n\square H$ is not $1$-tough.
\end{proof}

Considering the special case $n=1$ in Theorem~\ref{2.4}, we have the
following corollary, which is of independent interest.

\begin{cor}
An $1$-tough bipartite graph has a path factor.
\end{cor}

\section{Trees with perfect matchings}\label{s_4}

Let $C_n$ denote the cycle of order $n$. Results about the Hamiltonicity of Cartesian product graphs have been proved in several papers. For instance, the papers \cite{cff:09},\cite{dpp:05} and \cite{rb:73} have mentioned the following result.

\begin{thm}[\cite{rb:73}]\label{cnt}
Let $T$ be a tree.  If $n\geq\Delta(T)$, then $C_n\square T$ is Hamiltonian. \qed
\end{thm}

Motivated by Theorem~\ref{cnt}, we will prove the Hamiltonicity of $P_n\square T$. Before doing this we comment by the following lemma to  show that the assumption  $n\geq\Delta(T)$ in  Theorem~\ref{cnt}  is necessary.

\begin{lem}\label{1.2}
Let $G_1$ be a connected graph and $T$ be a tree. If $\Delta(T)>|V(G_1)|$, then the Cartesian product $G_1\square T$ is not $1$-tough.
\end{lem}
\begin{proof}
Find $v\in V(T)$ with $\deg(v)=\Delta(T)$, choose $S=\{u_v:u\in V(G_1)\}$ and note that $|S|=|V(G_1)|$. Now $c(G_1\square T -S)=\Delta(T)> |V(G_1)|=|S|$, which means that $G_1\square T$ is not $1$-tough.
\end{proof}

 Let $G$ be a graph with path factor $F$. Let $G_F$ be the graph with vertex set $F$ and two components $c_1,c_2\in F$ are  {\it adjacent} if there exist vertices $u\in c_1, v\in c_2$ such that $uv\in E(G)$.
In particular, if $T$ is a tree with  path factor $F$ then $T_F$ is a tree, deleting a leaf $c$ in $T_F$ yields a subtree of $T_F$, and $T-c$ is a subtree of $T$. Hence we have  the following lemma.

\begin{lem}\label{3.1}
For a tree $T$ with a $\{P_2,P_3\}$-factor $F$, there exists a component $c$ of $F$ such that $T-c$ is a tree with a $\{P_2,P_3\}$-factor $F-\{c\}$.\qed
\end{lem}

For $v\in V(T)$ let $B_v:=\{i_v(i+1)_v~|~1\leq i<n\}\subseteq E(P_n\square T).$
Now for $T=P_2$ and $V(T)=\{u,w\}$, the set
$\{1_u1_w\}\cup B_u \cup B_w\cup\{n_un_w\}$ of edges in $P_n\square T$
forms a Hamiltonian cycle, and call it the {\it standard} Hamiltonian cycle for $P_n\square P_2$.  To avoid confusions, the degree of vertex $v$ in $G$ will be denoted by $\deg_G(v)$.
To prove Theorem~\ref{main}(a), it is sufficient to find a Hamiltonian cycle of $P_n\square T$ where $n=|V(G_1)|$ and $T$ is a spanning tree of $G_2$ that contains  perfect matching $F$ of $G_2$. Note that $n\geq \Delta(G_2)\geq\Delta(T)$. For the convenience of proof, we state a stronger version as follows.

\begin{thm}\label{pn2}
Let $T$ be a tree with a perfect matching. If $n\geq\Delta(T)$, then there exists a Hamiltonian cycle of $P_n\square T$ which contains exactly $n-\deg_T(v)$ of the edges from the set $B_v$ for any vertex $v\in V(T)$. In particular, Theorem~\ref{main} {\it (a)} is proved.
\end{thm}
\begin{proof}
Apply induction on the number of vertices of $T$. For $T=P_2$, the standard Hamiltonian cycle for $P_n\square P_2$ satisfies the requirement since  $|B_v|=n-1=n-\deg_T(v)$ for $v\in V(P_2)$.

For a tree $T$ with a perfect matching $F$. By Lemma~\ref{3.1}, there exists a component (an edge) $c$ in $F$ such that $T-c$ is a tree with a perfect matching. Let $u_1\in c$ and $u_2\in V(T-c)$ such that $u_1$ and $u_2$ are adjacent. Let $H'$ be the standard Hamiltonian cycles of $P_n\square c$. Since the subtree $T'=T-c$ of $T$ has a perfect matching, $|V(T')|<|V(T)|$ and $n\geq\Delta(T)\geq\Delta(T')$,  by the induction hypothesis, there is a Hamiltonian cycle $H''$ of $P_n\square T'$ which contains exactly $n-\deg_{T'}(v)$ edges from the set $B_v$ for any vertex $v\in V(T')$. Since $n-\deg_{T'}(u_2)=n-(\deg_T(u_2)-1)\geq n-\Delta(T)+1\geq 1$, there exists a $j$ such that $j_{u_2}(j+1)_{u_2}\in H''$. Now
$$H=H'\cup H''\cup\{j_{u_1}j_{u_2},(j+1)_{u_1}(j+1)_{u_2}\}-\{j_{u_1}(j+1)_{u_1},j_{u_2}(j+1)_{u_2}\}$$
 is a Hamiltonian cycle of $P_n\square T$.

To check that $H$ satisfies the edge requirement, we only need to check those vertices in $T$ whose incident edges have been changed in the induction step, which are vertices $u_1$ and $u_2.$ For $u_1$, all the $n-1$ edges of $B_{u_1}$  are in the cycle $H'$. We delete one of them, so there are $n-2=n-\deg_T(u_1)$ edges from $B_{u_1}$ in $H$. For $u_2$, there are $n-\deg_{T'}(u_2)=n-(\deg_T(u_2)-1)$ edges from  $B_{u_2}$ in the cycle $H''$ by the induction hypothesis. We delete one of them, so there are $n-(\deg_T(u_2)-1)-1=n-\deg_T(u_2)$ edges from  $B_{u_2}$ in $H$.  This completes the proof.
\end{proof}

The paper \cite{cff:09} has proved that $G_1\square G_2$ is Hamiltonian when $G_1$ is traceable with $|V(G_1)|$ an even integer no less than $\Delta(G_2)-1$ and $G_2$ contains an even $2$-factor (i.e. a spanning subgraph consisting of even cycles). Since an even $2$-factor must contain an $1$-factor, so Theorem~\ref{main}{\it (a)} is a stronger result apart from the case $|V(G_1)|=\Delta(G_2)-1$.

\section{Graphs with path factors}\label{s4}

In this section, we construct a Hamiltonian cycle of $P_n\square G$ where $G$ is connected with a path factor and $n$ is an even integer with $n\geq 4\Delta(G)-2$.  By Lemma~\ref{p3.1}, $G$ has a $\{P_2, P_3\}$-factor $F$.
Let  $T$ be the spanning subtree of $G$ that contains $F$. It suffices to find a Hamiltonian cycle in $P_n\square T$.

For $v\in V(T)$, let $L_v=\{i_v(i+1)_v~|~i\equiv 0,1,3\Mod{4}\}, C_v=\{i_v(i+1)_v~|~i\equiv 0,2\Mod{4}\}, R_v=\{i_v(i+1)_v~|~i\equiv 1,2,3\Mod{4}\}$ denote three special subsets of the edge set $B_v$ described in the last section. For $G=P_3$ with $V(G)=\{u,v,w\}$ and $E(G)=\{uv,vw\}$, the set $\{1_u1_v\}\cup\{n_un_v,n_vn_w\}
\cup L_u\cup C_v\cup R_w
\cup\{i_ui_v:i\equiv 2,3\Mod{4}\}
\cup\{i_vi_w:i\equiv 0,1\Mod{4}\}$ of edges
forms a Hamiltonian cycle, and call it the {\it standard} Hamiltonian cycle for $P_n\square P_3$. See Figure 2 for the
standard Hamiltonian cycle for $P_{10}\square P_3$.
\bigskip

\begin{center}
\begin{tikzpicture}[rotate=270,line cap=round,line join=round,>=triangle 45,x=0.6cm,y=0.65cm]

\clip(0,0) rectangle (6,11.5);

\draw (4.5,0) node[anchor=north west] {$1_u$};
\draw (2.5,0) node[anchor=north west] {$1_v$};
\draw (0.5,0) node[anchor=north west] {$1_w$};
\draw (4.5,10) node[anchor=north west] {$10_u$};
\draw (2.5,10) node[anchor=north west] {$10_v$};
\draw (0.5,10) node[anchor=north west] {$10_w$};
\draw [line width=2pt] (1,1)-- (1,2);
\draw [line width=2pt] (1,3)-- (1,2);
\draw [line width=2pt] (1,3)-- (1,4);
\draw [line width=0.2pt,color=gray] (1,5)-- (1,4);
\draw [line width=2pt] (1,5)-- (1,6);
\draw [line width=2pt] (1,6)-- (1,7);
\draw [line width=2pt] (1,7)-- (1,8);
\draw[line width=0.2pt,color=gray] (1,8)-- (1,9);
\draw [line width=2pt] (1,9)-- (1,10);

\draw [line width=2pt] (3,1)-- (1,1);
\draw [line width=0.2pt,color=gray] (3,2)-- (1,2);
\draw [line width=0.2pt,color=gray] (3,3)-- (1,3);
\draw [line width=2pt] (3,4)-- (1,4);
\draw [line width=2pt] (3,5)-- (1,5);
\draw [line width=0.2pt,color=gray] (3,6)-- (1,6);
\draw [line width=0.2pt,color=gray] (3,7)-- (1,7);
\draw [line width=2pt] (3,8)-- (1,8);
\draw [line width=2pt] (3,9)-- (1,9);
\draw [line width=2pt] (3,10)-- (1,10);

\draw [line width=2pt] (3,1)-- (5,1);
\draw [line width=2pt] (3,2)-- (5,2);
\draw [line width=2pt] (3,3)-- (5,3);
\draw [line width=0.2pt,color=gray] (3,4)-- (5,4);
\draw [line width=0.2pt,color=gray] (3,5)-- (5,5);
\draw [line width=2pt] (3,6)-- (5,6);
\draw [line width=2pt] (3,7)-- (5,7);
\draw [line width=0.2pt,color=gray] (3,8)-- (5,8);
\draw [line width=0.2pt,color=gray] (3,9)-- (5,9);
\draw [line width=2pt] (3,10)-- (5,10);

\draw [line width=0.2pt,color=gray] (3,1)-- (3,2);
\draw [line width=2pt] (3,3)-- (3,2);
\draw [line width=0.2pt,color=gray] (3,3)-- (3,4);
\draw [line width=2pt] (3,5)-- (3,4);
\draw [line width=0.2pt,color=gray] (3,5)-- (3,6);
\draw [line width=2pt] (3,6)-- (3,7);
\draw [line width=0.2pt,color=gray] (3,7)-- (3,8);
\draw [line width=2pt] (3,8)-- (3,9);
\draw [line width=0.2pt,color=gray] (3,9)-- (3,10);

\draw [line width=2pt] (5,1)-- (5,2);
\draw [line width=0.2pt,color=gray] (5,3)-- (5,2);
\draw [line width=2pt] (5,3)-- (5,4);
\draw [line width=2pt] (5,5)-- (5,4);
\draw [line width=2pt] (5,5)-- (5,6);
\draw [line width=0.2pt,color=gray] (5,6)-- (5,7);
\draw [line width=2pt] (5,7)-- (5,8);
\draw [line width=2pt] (5,8)-- (5,9);
\draw [line width=2pt] (5,9)-- (5,10);

\begin{scriptsize}
\draw [fill=white] (1,1) circle (2.5pt);
\draw [fill=white] (1,2) circle (2.5pt);
\draw [fill=white] (1,3) circle (2.5pt);
\draw [fill=white] (1,4) circle (2.5pt);
\draw [fill=white] (1,5) circle (2.5pt);
\draw [fill=white] (1,6) circle (2.5pt);
\draw [fill=white] (1,7) circle (2.5pt);
\draw [fill=white] (1,8) circle (2.5pt);
\draw [fill=white] (1,9) circle (2.5pt);
\draw [fill=white] (1,10) circle (2.5pt);
\draw [fill=white] (3,1) circle (2.5pt);
\draw [fill=white] (3,2) circle (2.5pt);
\draw [fill=white] (3,3) circle (2.5pt);
\draw [fill=white] (3,4) circle (2.5pt);
\draw [fill=white] (3,5) circle (2.5pt);
\draw [fill=white] (3,6) circle (2.5pt);
\draw [fill=white] (3,7) circle (2.5pt);
\draw [fill=white] (3,8) circle (2.5pt);
\draw [fill=white] (3,9) circle (2.5pt);
\draw [fill=white] (3,10) circle (2.5pt);
\draw [fill=white] (5,1) circle (2.5pt);
\draw [fill=white] (5,2) circle (2.5pt);
\draw [fill=white] (5,3) circle (2.5pt);
\draw [fill=white] (5,4) circle (2.5pt);
\draw [fill=white] (5,5) circle (2.5pt);
\draw [fill=white] (5,6) circle (2.5pt);
\draw [fill=white] (5,7) circle (2.5pt);
\draw [fill=white] (5,8) circle (2.5pt);
\draw [fill=white] (5,9) circle (2.5pt);
\draw [fill=white] (5,10) circle (2.5pt);
\end{scriptsize}
\end{tikzpicture}
\medskip

{\bf Figure 2. } Standard Hamiltonian cycle for $P_{10}\square P_3$
\end{center}
\bigskip

By direct computation we have the following lemma.

\begin{lem}\label{lem4.1} For even integer $n$,
$|L_v\cap R_v|\geq |R_v\cap C_v|\geq |L_v\cap C_v|=\lceil \frac{n-4}{4}\rceil$.
\end{lem}

We define the {\it type} of a vertex $v$ in $T$ as follows.
$v$ has type $B$ (resp. $C$) if $v$ is in an edge in $F$  (resp. if $v$ is the middle vertex in a path of length $3$ in $F$). For the two endpoints of a path of length $3$ in $F$, we arbitrarily assign one endpoint of type $L$ and the other of type $R$.  Let $\delta_X=1$ if $X\in \{B, L, R\}$ and $\delta_X=2$ if $X=C$. Note that $\delta_X=\deg_c(v)$
for $c\in F$ and $v\in c$ of type $X$. The following is a stronger version of Theorem~\ref{main}(b).

\begin{thm}\label{pfactor}
Let $T$ be a connected graph with a $\{P_2, P_3\}$-factor $F$ and $n$ be an even integer. If $n\geq 4\Delta(T)-2$, then $P_n\square T$ contains a Hamiltonian cycle $H$ such that for any vertex $v\in V(T)$ of type $X\in \{B, L, C, R\}$,
we have $H\cap B_v\subseteq X_v$ and $|H\cap B_v|=|X_v|-\deg_T(v)+\delta_X$.  In particular, Theorem~\ref{main}(b) holds.
\end{thm}

\begin{proof} We prove by  induction on the number of vertices of $T$.  For $T=P_2$, any vertex $v$ of $P_2$ has type $B$ and the standard Hamiltonian cycle $H_1$ of $P_n\square P_2$ satisfies  $|H_1\cap B_v|=n-1=|B_v|-\deg_{P_2}(v)+1$ for vertex $v\in P_2$.
For $T=P_3$, a vertex $v$ of $P_3$ has type $X\in \{L, C, R\}$ and the standard Hamiltonian cycle $H_2$ of $P_n\square P_3$ satisfy $|H_2\cap B_v|= |X_v|=|X_v|-\deg_{P_3}(v)+\delta_X$.

Now assume $|V(T)|\geq 4$.  By Lemma~\ref{3.1}, there exists a component $c$ of $F$ such that $T-c$ is a tree with the path factor $F-\{c\}$. Let $u_1\in c$ and $u_2\in V(T-c)$ such that $u_1$ and $u_2$ are adjacent.
Assume $u_1$ has type $X$ and $u_2$ has type $Y$.
Let $H'$ be the standard Hamiltonian cycle of $P_n\square c$ and $P_n\square T-c$ contains
a Hamiltonian  cycle $H''$
  that satisfies $H''\cap B_{u_2}\subseteq Y_{u_2}$ and $|H_2\cap B_{u_2}|=|Y_{u_2}|-(\deg_T(u_2)-1)+\delta_Y$ by induction hypothesis. Referring to Lemma~\ref{lem4.1}, we have $|H_2\cap B_{u_2}\cap X_{u_2}|\geq |Y_{u_2}\cap X_{u_2}|-(\deg_T(u_2)-1)+\delta_Y\geq \lceil \frac{n-4}{4}\rceil-\deg_T(u_2)+2\geq  \lceil \frac{4\Delta(T)-6}{4}\rceil-\deg_T(u_2)+2\geq 1.$
Pick $j_{u_2}(j+1)_{u_2}\in H_2\cap B_{u_2}\cap X_{u_2}$ and then $j_{u_1}(j+1)_{u_1}\in X_{u_1}\subseteq H'$.
Now
$$H=H'\cup H''\cup\{j_{u_1}j_{u_2},(j+1)_{u_1}(j+1)_{u_2}\}-\{j_{u_1}(j+1)_{u_1},j_{u_2}(j+1)_{u_2}\}$$
is a Hamiltonian cycle of $P_n\square T$.

To check $H$ satisfies the edge requirements, we only need to check for $v\in \{u_1, u_2\}$.
This follows from $|H\cap X_{u_1}|= |H'\cap X_{u_1}|-1=|X_{u_1}|-1=|X_{u_1}|-\deg_T(u_1)+\delta_X$ and
$|H\cap Y_{u_2}|= |H''\cap Y_{u_2}|-1=|Y_{u_2}|-(\deg_T(u_2)-1)+\delta_Y-1=|Y_{u_2}|-\deg_T(u_2)+\delta_Y.$
\end{proof}

This theorem can be compared to a result in \cite{cff:09}. The authors considered a sub-class of trees called $1$-pendant trees. Let $T$ be a $1$-pendant tree that contains a path factor and $n$ an odd integer no less than $2\Delta(T)-2$ . Now $P_n\square T$ has a cycle omitting at most $s$ vertices, where $s$ is the number of odd components of the path factor. Since some of $P_n\square T$ is not $1$-tough when $n$ is odd, so practically they gave a good direction to find the long cycles of Cartesian product graphs which are not $1$-tough.

\section{Proofs of the corollaries}\label{s2}

\noindent {\bf Proof of Corollary~\ref{c2}:}\\
Let $S$ be a vertex subset of $V(G_2)$. Now the number of edges between $S$ and the set of isolated vertices of $G_2-S$
is at least $i(G_2-S)\delta(G_2)$ and is at most $|S|\Delta(G_2)$.
Since $2\delta(G_2)\geq \Delta(G_2)$, we have $i(G_2-S)\leq 2|S|$ for all $S\subseteq V(G_2)$. By Proposition~\ref{p3.2}, $G_2$ has a path factor and by Theorem~{\ref{main}}(b) we complete the proof. \qed
\bigskip

\noindent {\bf Proof of Corollary~\ref{thm0}:}\\
This is immediate by applying Lemma~\ref{l2.3} to Theorem~\ref{main}(b). \qed
\bigskip

\noindent {\bf Proof of Corollary~\ref{thm2}:} \\
$(1)\Rightarrow (2)$ is clear. $(2)\Rightarrow (3)$ is from Lemma~\ref{1.2}.
$(3)\Rightarrow (1)$ is from Theorem~\ref{main}(a).\qed
\bigskip

\noindent {\bf Proof of Corollary~\ref{thm1}:}\\
$(1)\Rightarrow (2)$ is clear. $(2)\Rightarrow (3)$  is from Theorem~\ref{2.4}. $(3)\Rightarrow (1)$  is from Theorem~\ref{main}{(b)}.\qed
\bigskip

To show that the assumption $n\geq 4\Delta(H)-2$ in Corollary~\ref{thm1} can not be replaced by $n\geq\Delta(H)$, we provide a $1$-tough non-Hamiltonian graph $P_n\square T$ such that $T$ is a tree with a path factor and $n=\Delta(T)+1$.
\par
Let $T_1$ be a tree with vertex set $V(T_1)=\{1,2,3,4,5,6,7,8\}$ and edge set $E(T_1)=\{12,23,34,45,26,37,48\}$.

\begin{prop}\label{t1}
The graph $G=P_4\square T_1$ is $1$-tough but not Hamiltonian.
\end{prop}
\begin{proof}
If $G$ is Hamiltonian, the edges incident to degree two vertices of $G$ must contained in each Hamiltonian cycle. Therefore the edges $1_11_2,1_21_6,1_31_7,1_41_5,$ $1_41_8,1_12_1,1_52_5,1_62_6,1_72_7,1_82_8,3_14_1,3_54_5,3_64_6,3_74_7,3_82_4,4_14_2,4_24_6,4_34_7,$\\$4_44_5$ and $4_44_8$ (thick black edges in Figure~3(a)) are chosen. Since each of the vertices $1_2,1_4,$ $4_2$ and $4_4$ is already incident to two chosen edges, the four edges $1_21_3,1_31_4,4_24_3,4_34_4$ (dotted edges in Figure~3(b)) can not be chosen. Furthermore, this tells that the edges $1_32_3,3_34_3$ need to be chosen as shown in Figure~3(b). At this time, at least one of $2_22_3,2_32_4$ can not be chosen to complete the Hamiltonian cycle. Without loss of generality, says the edge $2_22_3$ (dashed edges in Figure~3(b)) has not been chosen. Now each of two internal disjoint paths from $2_2$ to $2_3$ in the Hamiltonian cycle contains the edge $3_23_3$, a contradiction. Hence $G$ is not Hamiltonian.
\par
Next we show that $G$ is $1$-tough. As $G$ depicted in Figure~3(c), there exists a cycle $C$ of order $30$ in $G$ such that  $V(G-C)=\{3_5,4_5\}$ and $3_54_5$ is an edge of $G$ that is incident to $3$ vertices $2_5,3_4,4_4$ of $C$. For a vertex set $S$, there are $3$ cases for $G-S$ to discuss : The set $S\cap\{3_5,4_5\}$ is non-empty; The set $S\cap\{3_5,4_5\}$ is empty and $\{2_5,3_4,4_4\}\subseteq S$; The set $S\cap\{3_5,4_5\}$ is empty and $\{2_5,3_4,4_4\}\not\subseteq S$.
\par
If the set $S\cap\{3_5,4_5\}$ is non-empty, then $c(\{3_5,4_5\}-(S\cap\{3_5,4_5\}))\leq|S\cap\{3_5,4_5\}|$. On the other hand, $c(S\cap C)\leq|C-(S\cap C)|$ since $C$ is $1$-tough. Because $G-S\subseteq(C-(S\cap C))\cup(\{3_5,4_5\}-S\cap\{3_5,4_5\})$, we conclude that $c(G-S)\leq c(C-(S\cap C))+c(\{3_5,4_5\}-(S\cap\{3_5,4_5\}))\leq|S\cap C|+|S\cap\{3_5,4_5\}|=|S|$ for all $S$ such that $S\cap\{3_5,4_5\}$ is non-empty.
\par
If the set $S\cap\{3_5,4_5\}$ is empty and $\{2_5,3_4,4_4\}\subseteq S$, then the subgraph induced by $\{3_5,4_5\}$ is a component of $G-S$. As depicted in Figure~3(d), the subgraph $G_1$ of $G$ induced by $V(G)-\{3_5,4_5,2_5,3_4,4_4\}$ contains a spanning tree such that all vertices has degree at most $2$ except an only one degree $3$ vertex. This implies $c(G_1-S')\leq|S'|+2$ for $S'=S-\{2_5,3_4,4_4\}$. Therefore,
$c(G-S)=c(G_1-S')+1\leq|S'|+3=|S|$ for all $S$ such that the subgraph induced by $\{3_5,4_5\}$ is a component of $G-S$.
\par
If the set $S\cap\{3_5,4_5\}$ is empty and $\{2_5,3_4,4_4\}\not\subseteq S$, then $S\subseteq C$ and the edge $3_54_5$ is adjacent to some vertices of $C-S$. Therefore, $c(G-S)\leq c(C-S)$ for all such $S$. Since the cycle $C$ is $1$-tough, $c(C-S)\leq|S|$. Hence $c(G-S)\leq c(C-S)\leq|S|$.
\par
In conclusion, $c(G-S)\leq|S|$ for all $S\subseteq V(G)$ which means $G$ is $1$-tough.

\end{proof}

\begin{center}
\begin{minipage}{0.49\textwidth}
\begin{tikzpicture}[line cap=round,line join=round,>=triangle 45,x=0.54cm,y=0.27cm]

\clip(-12,-1) rectangle (-1,17);
\draw [line width=2.pt] (-11,3)-- (-11,7);
\draw [line width=0.2pt,color=gray] (-11,7)-- (-11,11);
\draw [line width=2.pt] (-11,11)-- (-11,15);
\draw [line width=0.2pt,color=gray] (-9,2)-- (-9,6);
\draw [line width=0.2pt,color=gray] (-9,6)-- (-9,10);
\draw [line width=0.2pt,color=gray] (-9,10)-- (-9,14);
\draw [line width=0.2pt,color=gray] (-7,2)-- (-7,6);
\draw [line width=0.2pt,color=gray] (-7,6)-- (-7,10);
\draw [line width=0.2pt,color=gray] (-7,10)-- (-7,14);
\draw [line width=0.2pt,color=gray] (-5,2)-- (-5,6);
\draw [line width=0.2pt,color=gray] (-5,6)-- (-5,10);
\draw [line width=0.2pt,color=gray] (-5,10)-- (-5,14);
\draw [line width=2.pt] (-10,0)-- (-10,4);
\draw [line width=0.2pt,color=gray] (-10,4)-- (-10,8);
\draw [line width=2.pt] (-10,8)-- (-10,12);
\draw [line width=2.pt](-8,0)-- (-8,4);
\draw [line width=0.2pt,color=gray] (-8,4)-- (-8,8);
\draw [line width=2.pt] (-8,8)-- (-8,12);
\draw [line width=2.pt] (-3,1)-- (-3,5);
\draw [line width=0.2pt,color=gray] (-3,5)-- (-3,9);
\draw [line width=2.pt] (-3,9)-- (-3,13);
\draw [line width=2.pt] (-4,4)-- (-4,8);
\draw [line width=0.2pt,color=gray] (-4,8)-- (-4,12);
\draw [line width=2.pt] (-4,12)-- (-4,16);
\draw [line width=2.pt] (-11,3)-- (-9,2);
\draw [line width=0.2pt,color=gray] (-11,7)-- (-9,6);
\draw [line width=0.2pt,color=gray] (-11,11)-- (-9,10);
\draw [line width=2.pt] (-11,15)-- (-9,14);
\draw [line width=2.pt] (-10,0)-- (-9,2);
\draw [line width=0.2pt,color=gray] (-10,4)-- (-9,6);
\draw [line width=0.2pt,color=gray] (-10,8)-- (-9,10);
\draw [line width=2.pt] (-10,12)-- (-9,14);
\draw [line width=2.pt] (-8,0)-- (-7,2);
\draw [line width=0.2pt,color=gray] (-8,4)-- (-7,6);
\draw [line width=0.2pt,color=gray] (-8,8)-- (-7,10);
\draw [line width=2.pt] (-8,12)-- (-7,14);
\draw [line width=0.2pt,color=gray] (-9,2)-- (-7,2);
\draw [line width=0.2pt,color=gray] (-9,6)-- (-7,6);
\draw [line width=0.2pt,color=gray] (-9,10)-- (-7,10);
\draw [line width=0.2pt,color=gray] (-9,14)-- (-7,14);
\draw [line width=0.2pt,color=gray] (-5,2)-- (-7,2);
\draw [line width=0.2pt,color=gray] (-5,6)-- (-7,6);
\draw [line width=0.2pt,color=gray] (-5,10)-- (-7,10);
\draw [line width=0.2pt,color=gray] (-5,14)-- (-7,14);
\draw [line width=2.pt] (-5,2)-- (-4,4);
\draw [line width=0.2pt,color=gray] (-5,6)-- (-4,8);
\draw [line width=0.2pt,color=gray] (-5,10)-- (-4,12);
\draw [line width=2.pt] (-5,14)-- (-4,16);
\draw [line width=2.pt] (-5,2)-- (-3,1);
\draw [line width=0.2pt,color=gray] (-5,6)-- (-3,5);
\draw [line width=0.2pt,color=gray] (-5,10)-- (-3,9);
\draw [line width=2.pt] (-5,14)-- (-3,13);

\begin{scriptsize}
\draw [fill=white] (-11,3) circle (2.5pt);
\draw [fill=white] (-9,2) circle (2.5pt);
\draw [fill=white] (-10,0) circle (2.5pt);
\draw [fill=white] (-7,2) circle (2.5pt);
\draw [fill=white] (-5,2) circle (2.5pt);
\draw [fill=white] (-3,1) circle (2.5pt);
\draw [fill=white] (-4,4) circle (2.5pt);
\draw [fill=white] (-10,4) circle (2.5pt);
\draw [fill=white] (-11,7) circle (2.5pt);
\draw [fill=white] (-9,6) circle (2.5pt);
\draw [fill=white] (-7,6) circle (2.5pt);
\draw [fill=white] (-5,6) circle (2.5pt);
\draw [fill=white] (-4,8) circle (2.5pt);
\draw [fill=white] (-3,5) circle (2.5pt);
\draw [fill=white] (-8,4) circle (2.5pt);
\draw [fill=white] (-11,11) circle (2.5pt);
\draw [fill=white] (-11,15) circle (2.5pt);
\draw [fill=white] (-9,10) circle (2.5pt);
\draw [fill=white] (-9,14) circle (2.5pt);
\draw [fill=white] (-7,10) circle (2.5pt);
\draw [fill=white] (-7,14) circle (2.5pt);
\draw [fill=white] (-10,8) circle (2.5pt);
\draw [fill=white] (-10,12) circle (2.5pt);
\draw [fill=white] (-8,8) circle (2.5pt);
\draw [fill=white] (-8,12) circle (2.5pt);
\draw [fill=white] (-5,10) circle (2.5pt);
\draw [fill=white] (-4,12) circle (2.5pt);
\draw [fill=white] (-3,9) circle (2.5pt);
\draw [fill=white] (-5,14) circle (2.5pt);
\draw [fill=white] (-4,16) circle (2.5pt);
\draw [fill=white] (-3,13) circle (2.5pt);
\draw [fill=white] (-8,0) circle (2.5pt);
\end{scriptsize}
\end{tikzpicture}
\\(a) Edges with degree $2$ endpoints
\end{minipage}
\begin{minipage}{0.49\textwidth}
\begin{tikzpicture}[line cap=round,line join=round,>=triangle 45,x=0.54cm,y=0.27cm]

\clip(-12,-1) rectangle (-1,17);
\draw [line width=2.pt] (-11,3)-- (-11,7);
\draw [line width=0.2pt,color=gray] (-11,7)-- (-11,11);
\draw [line width=2.pt] (-11,11)-- (-11,15);
\draw [line width=0.2pt,color=gray] (-9,2)-- (-9,6);
\draw [line width=0.2pt,color=gray] (-9,6)-- (-9,10);
\draw [line width=0.2pt,color=gray] (-9,10)-- (-9,14);
\draw [line width=2.pt] (-7,2)-- (-7,6);
\draw [line width=0.2pt,color=gray] (-7,6)-- (-7,10);
\draw [line width=2.pt] (-7,10)-- (-7,14);
\draw [line width=0.2pt,color=gray] (-5,2)-- (-5,6);
\draw [line width=0.2pt,color=gray] (-5,6)-- (-5,10);
\draw [line width=0.2pt,color=gray] (-5,10)-- (-5,14);
\draw [line width=2.pt] (-10,0)-- (-10,4);
\draw [line width=0.2pt,color=gray] (-10,4)-- (-10,8);
\draw [line width=2.pt] (-10,8)-- (-10,12);
\draw [line width=2.pt](-8,0)-- (-8,4);
\draw [line width=0.2pt,color=gray] (-8,4)-- (-8,8);
\draw [line width=2.pt] (-8,8)-- (-8,12);
\draw [line width=2.pt] (-3,1)-- (-3,5);
\draw [line width=0.2pt,color=gray] (-3,5)-- (-3,9);
\draw [line width=2.pt] (-3,9)-- (-3,13);
\draw [line width=2.pt] (-4,4)-- (-4,8);
\draw [line width=0.2pt,color=gray] (-4,8)-- (-4,12);
\draw [line width=2.pt] (-4,12)-- (-4,16);
\draw [line width=2.pt] (-11,3)-- (-9,2);
\draw [line width=0.2pt,color=gray] (-11,7)-- (-9,6);
\draw [line width=0.2pt,color=gray] (-11,11)-- (-9,10);
\draw [line width=2.pt] (-11,15)-- (-9,14);
\draw [line width=2.pt] (-10,0)-- (-9,2);
\draw [line width=0.2pt,color=gray] (-10,4)-- (-9,6);
\draw [line width=0.2pt,color=gray] (-10,8)-- (-9,10);
\draw [line width=2.pt] (-10,12)-- (-9,14);
\draw [line width=2.pt] (-8,0)-- (-7,2);
\draw [line width=0.2pt,color=gray] (-8,4)-- (-7,6);
\draw [line width=0.2pt,color=gray] (-8,8)-- (-7,10);
\draw [line width=2.pt] (-8,12)-- (-7,14);
\draw [line width=0.2pt,color=gray, dotted] (-9,2)-- (-7,2);
\draw [line width=0.2pt,color=gray] (-9,6)-- (-7,6);
\draw [line width=0.2pt,color=gray, dashed] (-9,10)-- (-7,10);
\draw [line width=0.2pt,color=gray, dotted] (-9,14)-- (-7,14);
\draw [line width=0.2pt,color=gray, dotted] (-5,2)-- (-7,2);
\draw [line width=0.2pt,color=gray] (-5,6)-- (-7,6);
\draw [line width=0.2pt,color=gray] (-5,10)-- (-7,10);
\draw [line width=0.2pt,color=gray, dotted] (-5,14)-- (-7,14);
\draw [line width=2.pt] (-5,2)-- (-4,4);
\draw [line width=0.2pt,color=gray] (-5,6)-- (-4,8);
\draw [line width=0.2pt,color=gray] (-5,10)-- (-4,12);
\draw [line width=2.pt] (-5,14)-- (-4,16);
\draw [line width=2.pt] (-5,2)-- (-3,1);
\draw [line width=0.2pt,color=gray] (-5,6)-- (-3,5);
\draw [line width=0.2pt,color=gray] (-5,10)-- (-3,9);
\draw [line width=2.pt] (-5,14)-- (-3,13);

\begin{scriptsize}
\draw [fill=white] (-11,3) circle (2.5pt);
\draw [fill=white] (-9,2) circle (2.5pt);
\draw [fill=white] (-10,0) circle (2.5pt);
\draw [fill=white] (-7,2) circle (2.5pt);
\draw [fill=white] (-5,2) circle (2.5pt);
\draw [fill=white] (-3,1) circle (2.5pt);
\draw [fill=white] (-4,4) circle (2.5pt);
\draw [fill=white] (-10,4) circle (2.5pt);
\draw [fill=white] (-11,7) circle (2.5pt);
\draw [fill=white] (-9,6) circle (2.5pt);
\draw [fill=white] (-7,6) circle (2.5pt);
\draw [fill=white] (-5,6) circle (2.5pt);
\draw [fill=white] (-4,8) circle (2.5pt);
\draw [fill=white] (-3,5) circle (2.5pt);
\draw [fill=white] (-8,4) circle (2.5pt);
\draw [fill=white] (-11,11) circle (2.5pt);
\draw [fill=white] (-11,15) circle (2.5pt);
\draw [fill=white] (-9,10) circle (2.5pt);
\draw [fill=white] (-9,14) circle (2.5pt);
\draw [fill=white] (-7,10) circle (2.5pt);
\draw [fill=white] (-7,14) circle (2.5pt);
\draw [fill=white] (-10,8) circle (2.5pt);
\draw [fill=white] (-10,12) circle (2.5pt);
\draw [fill=white] (-8,8) circle (2.5pt);
\draw [fill=white] (-8,12) circle (2.5pt);
\draw [fill=white] (-5,10) circle (2.5pt);
\draw [fill=white] (-4,12) circle (2.5pt);
\draw [fill=white] (-3,9) circle (2.5pt);
\draw [fill=white] (-5,14) circle (2.5pt);
\draw [fill=white] (-4,16) circle (2.5pt);
\draw [fill=white] (-3,13) circle (2.5pt);
\draw [fill=white] (-8,0) circle (2.5pt);
\end{scriptsize}

\draw (-9.2,8.2) node[anchor=north west] {$3_2$};
\draw (-7.2,8.2) node[anchor=north west] {$3_3$};

\end{tikzpicture}
\\(b) Edges which need to be chosen
\end{minipage}
\\
\begin{minipage}{0.49\textwidth}
\begin{tikzpicture}[line cap=round,line join=round,>=triangle 45,x=0.54cm,y=0.27cm]

\clip(-12,-1) rectangle (-1,17);
\draw [line width=2.pt] (-11,3)-- (-11,7);
\draw [line width=2.pt] (-11,7)-- (-11,11);
\draw [line width=2.pt] (-11,11)-- (-11,15);
\draw [line width=0.2pt,color=gray] (-9,2)-- (-9,6);
\draw [line width=0.2pt,color=gray] (-9,6)-- (-9,10);
\draw [line width=0.2pt,color=gray] (-9,10)-- (-9,14);
\draw [line width=0.2pt,color=gray] (-7,2)-- (-7,6);
\draw [line width=0.2pt,color=gray] (-7,6)-- (-7,10);
\draw [line width=2.pt] (-7,10)-- (-7,14);
\draw [line width=0.2pt,color=gray] (-5,2)-- (-5,6);
\draw [line width=2.pt] (-5,6)-- (-5,10);
\draw [line width=0.2pt,color=gray] (-5,10)-- (-5,14);
\draw [line width=2.pt] (-10,0)-- (-10,4);
\draw [line width=0.2pt,color=gray] (-10,4)-- (-10,8);
\draw [line width=2.pt] (-10,8)-- (-10,12);
\draw [line width=2.pt] (-8,0)-- (-8,4);
\draw [line width=2.pt] (-8,4)-- (-8,8);
\draw [line width=2.pt] (-8,8)-- (-8,12);
\draw [line width=3.pt] (-3,1)-- (-3,5);
\draw [line width=0.2pt,color=gray] (-3,5)-- (-3,9);
\draw [line width=2.pt] (-3,9)-- (-3,13);
\draw [line width=2.pt] (-4,4)-- (-4,8);
\draw [line width=2.pt] (-4,8)-- (-4,12);
\draw [line width=2.pt] (-4,12)-- (-4,16);
\draw [line width=2.pt] (-11,3)-- (-9,2);
\draw [line width=0.2pt,color=gray] (-11,7)-- (-9,6);
\draw [line width=0.2pt,color=gray] (-11,11)-- (-9,10);
\draw [line width=2.pt] (-11,15)-- (-9,14);
\draw [line width=2.pt] (-10,0)-- (-9,2);
\draw [line width=2.pt] (-10,4)-- (-9,6);
\draw [line width=2.pt] (-10,8)-- (-9,10);
\draw [line width=2.pt] (-10,12)-- (-9,14);
\draw [line width=2.pt] (-8,0)-- (-7,2);
\draw [line width=0.2pt,color=gray] (-8,4)-- (-7,6);
\draw [line width=0.2pt,color=gray] (-8,8)-- (-7,10);
\draw [line width=2.pt] (-8,12)-- (-7,14);
\draw [line width=0.2pt,color=gray] (-9,2)-- (-7,2);
\draw [line width=2.pt] (-9,6)-- (-7,6);
\draw [line width=2.pt] (-9,10)-- (-7,10);
\draw [line width=0.2pt,color=gray] (-9,14)-- (-7,14);
\draw [line width=2.pt] (-5,2)-- (-7,2);
\draw [line width=2.pt] (-5,6)-- (-7,6);
\draw [line width=0.2pt,color=gray] (-5,10)-- (-7,10);
\draw [line width=0.2pt,color=gray] (-5,14)-- (-7,14);
\draw [line width=2.pt] (-5,2)-- (-4,4);
\draw [line width=0.2pt,color=gray] (-5,6)-- (-4,8);
\draw [line width=0.2pt,color=gray] (-5,10)-- (-4,12);
\draw [line width=2.pt] (-5,14)-- (-4,16);
\draw [line width=0.2pt,color=gray] (-5,2)-- (-3,1);
\draw [line width=0.2pt,color=gray] (-5,6)-- (-3,5);
\draw [line width=2.pt] (-5,10)-- (-3,9);
\draw [line width=2.pt] (-5,14)-- (-3,13);

\begin{scriptsize}
\draw [fill=white] (-11,3) circle (2.5pt);
\draw [fill=white] (-9,2) circle (2.5pt);
\draw [fill=white] (-10,0) circle (2.5pt);
\draw [fill=white] (-7,2) circle (2.5pt);
\draw [fill=white] (-5,2) circle (2.5pt);
\draw [fill=black] (-3,1) circle (2.5pt);
\draw [fill=white] (-4,4) circle (2.5pt);
\draw [fill=white] (-10,4) circle (2.5pt);
\draw [fill=white] (-11,7) circle (2.5pt);
\draw [fill=white] (-9,6) circle (2.5pt);
\draw [fill=white] (-7,6) circle (2.5pt);
\draw [fill=white] (-5,6) circle (2.5pt);
\draw [fill=white] (-4,8) circle (2.5pt);
\draw [fill=black] (-3,5) circle (2.5pt);
\draw [fill=white] (-8,4) circle (2.5pt);
\draw [fill=white] (-11,11) circle (2.5pt);
\draw [fill=white] (-11,15) circle (2.5pt);
\draw [fill=white] (-9,10) circle (2.5pt);
\draw [fill=white] (-9,14) circle (2.5pt);
\draw [fill=white] (-7,10) circle (2.5pt);
\draw [fill=white] (-7,14) circle (2.5pt);
\draw [fill=white] (-10,8) circle (2.5pt);
\draw [fill=white] (-10,12) circle (2.5pt);
\draw [fill=white] (-8,8) circle (2.5pt);
\draw [fill=white] (-8,12) circle (2.5pt);
\draw [fill=white] (-5,10) circle (2.5pt);
\draw [fill=white] (-4,12) circle (2.5pt);
\draw [fill=white] (-3,9) circle (2.5pt);
\draw [fill=white] (-5,14) circle (2.5pt);
\draw [fill=white] (-4,16) circle (2.5pt);
\draw [fill=white] (-3,13) circle (2.5pt);
\draw [fill=white] (-8,0) circle (2.5pt);
\end{scriptsize}

\draw (-3,1) node[anchor=north west] {$4_5$};
\draw (-3,5) node[anchor=north west] {$3_5$};
\draw (-3,9) node[anchor=north west] {$2_5$};
\draw (-5.75,6) node[anchor=north west] {$3_4$};
\draw (-5.75,2) node[anchor=north west] {$4_4$};
\end{tikzpicture}
\\(c) A cycle $C$ and the edge $3_54_5$
\end{minipage}
\begin{minipage}{0.49\textwidth}
\begin{tikzpicture}[line cap=round,line join=round,>=triangle 45,x=0.54cm,y=0.27cm]

\clip(-12,-1) rectangle (-1,17);
\draw [line width=2.pt] (-11,3)-- (-11,7);
\draw [line width=0.2pt,color=gray] (-11,7)-- (-11,11);
\draw [line width=2.pt] (-11,11)-- (-11,15);
\draw [line width=0.2pt,color=gray] (-9,2)-- (-9,6);
\draw [line width=0.2pt,color=gray] (-9,6)-- (-9,10);
\draw [line width=0.2pt,color=gray] (-9,10)-- (-9,14);
\draw [line width=2.pt] (-7,2)-- (-7,6);
\draw [line width=0.2pt,color=gray] (-7,6)-- (-7,10);
\draw [line width=0.2pt,color=gray] (-7,10)-- (-7,14);
\draw [line width=2.pt] (-5,10)-- (-5,14);
\draw [line width=2.pt] (-10,0)-- (-10,4);
\draw [line width=2.pt] (-10,4)-- (-10,8);
\draw [line width=2.pt] (-10,8)-- (-10,12);
\draw [line width=2.pt] (-8,0)-- (-8,4);
\draw [line width=2.pt] (-8,4)-- (-8,8);
\draw [line width=2.pt] (-8,8)-- (-8,12);
\draw [line width=2.pt] (-4,4)-- (-4,8);
\draw [line width=2.pt] (-4,8)-- (-4,12);
\draw [line width=2.pt] (-4,12)-- (-4,16);
\draw [line width=2.pt] (-11,3)-- (-9,2);
\draw [line width=2.pt] (-11,7)-- (-9,6);
\draw [line width=2.pt] (-11,11)-- (-9,10);
\draw [line width=2.pt] (-11,15)-- (-9,14);
\draw [line width=2.pt] (-10,0)-- (-9,2);
\draw [line width=0.2pt,color=gray] (-10,4)-- (-9,6);
\draw [line width=0.2pt,color=gray] (-10,8)-- (-9,10);
\draw [line width=2.pt] (-10,12)-- (-9,14);
\draw [line width=2.pt] (-8,0)-- (-7,2);
\draw [line width=0.2pt,color=gray] (-8,4)-- (-7,6);
\draw [line width=0.2pt,color=gray] (-8,8)-- (-7,10);
\draw [line width=2.pt] (-8,12)-- (-7,14);
\draw [line width=0.2pt,color=gray] (-9,2)-- (-7,2);
\draw [line width=2.pt] (-9,6)-- (-7,6);
\draw [line width=2.pt] (-9,10)-- (-7,10);
\draw [line width=0.2pt,color=gray] (-9,14)-- (-7,14);
\draw [line width=2.pt] (-5,10)-- (-7,10);
\draw [line width=0.2pt,color=gray] (-5,14)-- (-7,14);
\draw [line width=0.2pt,color=gray] (-5,10)-- (-4,12);
\draw [line width=2.pt] (-5,14)-- (-4,16);
\draw [line width=2.pt] (-5,14)-- (-3,13);

\begin{scriptsize}
\draw [fill=white] (-11,3) circle (2.5pt);
\draw [fill=white] (-9,2) circle (2.5pt);
\draw [fill=white] (-10,0) circle (2.5pt);
\draw [fill=white] (-7,2) circle (2.5pt);
\draw [fill=white] (-4,4) circle (2.5pt);
\draw [fill=white] (-10,4) circle (2.5pt);
\draw [fill=white] (-11,7) circle (2.5pt);
\draw [fill=white] (-9,6) circle (2.5pt);
\draw [fill=white] (-7,6) circle (2.5pt);
\draw [fill=white] (-4,8) circle (2.5pt);
\draw [fill=white] (-8,4) circle (2.5pt);
\draw [fill=white] (-11,11) circle (2.5pt);
\draw [fill=white] (-11,15) circle (2.5pt);
\draw [fill=white] (-9,10) circle (2.5pt);
\draw [fill=white] (-9,14) circle (2.5pt);
\draw [fill=white] (-7,10) circle (2.5pt);
\draw [fill=white] (-7,14) circle (2.5pt);
\draw [fill=white] (-10,8) circle (2.5pt);
\draw [fill=white] (-10,12) circle (2.5pt);
\draw [fill=white] (-8,8) circle (2.5pt);
\draw [fill=white] (-8,12) circle (2.5pt);
\draw [fill=white] (-5,10) circle (2.5pt);
\draw [fill=white] (-4,12) circle (2.5pt);
\draw [fill=white] (-5,14) circle (2.5pt);
\draw [fill=white] (-4,16) circle (2.5pt);
\draw [fill=white] (-3,13) circle (2.5pt);
\draw [fill=white] (-8,0) circle (2.5pt);
\end{scriptsize}
\end{tikzpicture}
\\(d) Graph $G_1$ and its spanning tree
\end{minipage}
\bigskip

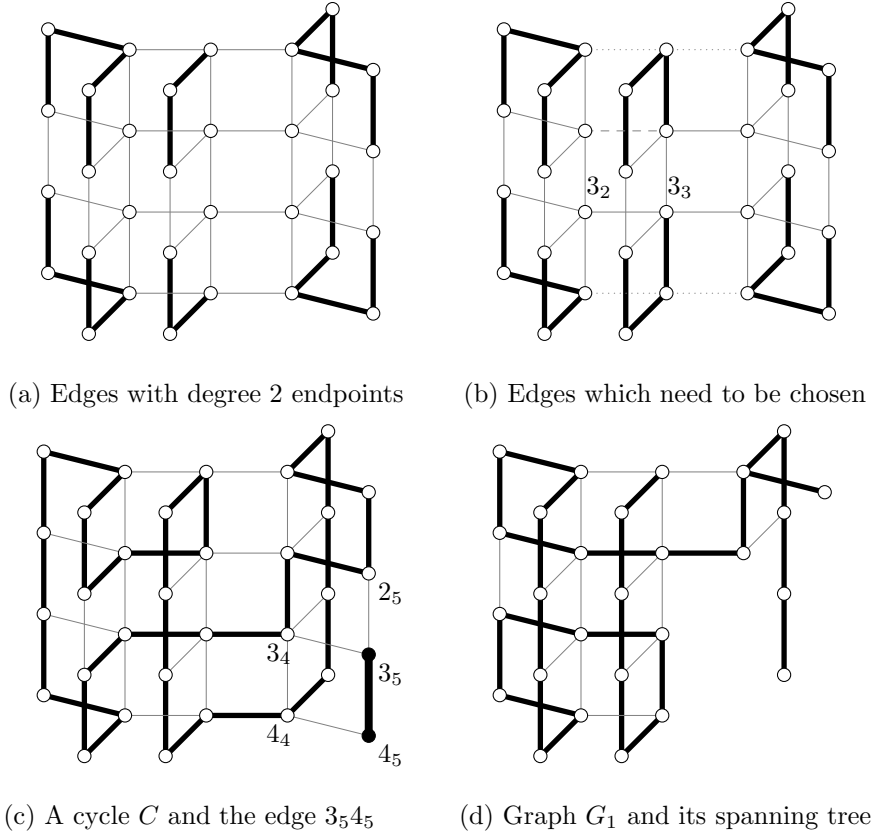
{\bf Figure 3. } The graph $P_4\square T_1$ and its subgraphs
\end{center}
\bigskip

\section{Concluding remarks}\label{concluding}

For $G_1$ traceable, $G_2$ containing a path factor, $|V(G_1)|$ even and $|V(G_1)|\geq 4\Delta(G_2)-2$, we construct a Hamiltonian cycle for $G_1\square G_2$ in Theorem~\ref{main}(b). On the other hand, Proposition~\ref{t1} shows that the above assumption $|V(G_1)|\geq 4\Delta(G_2)-2$ can not be extended to $|V(G_1)|\geq \Delta(G_2).$
In general, we further conjecture that the assumption $|V(G_1)|\geq 4\Delta(G_2)-4$ is not sufficient for $G_1\square G_2$ to be Hamiltonian.

\begin{conj}
For $k\geq 3$, there is a connected graph $G$ with a path factor such that $\Delta(G)=k$ and $P_{4k-4}\square G$ is not Hamiltonian.
\end{conj}

If a bipartite graph has an equal size bipartition, we call it {\it balanced}. For $n$ being an odd integer and $G$ a bipartite graph, the graph $P_n\square G$ is also possible to be Hamiltonian if it is balanced. For instance, let $V(G_2)=\{1,2,3,4,5,6\}$ and $E(G_2)=\{12,23,34,25,36\}$. The graph $P_5\square G_2$ is Hamiltonian as depicted in Figure~4.
\bigskip

\begin{center}
\begin{tikzpicture}[line cap=round,line join=round,>=triangle 45,x=1cm,y=0.35cm]

\clip(-11,-1) rectangle (-1,21);

\draw [line width=0.2pt,color=gray] (-7,2)-- (-7,6);
\draw [line width=0.2pt,color=gray] (-7,6)-- (-7,10);
\draw [line width=0.2pt,color=gray] (-7,10)-- (-7,14);
\draw [line width=0.2pt,color=gray] (-7,14)-- (-7,18);
\draw [line width=0.2pt,color=gray] (-5,2)-- (-5,6);
\draw [line width=0.2pt,color=gray] (-5,6)-- (-5,10);
\draw [line width=0.2pt,color=gray] (-5,10)-- (-5,14);
\draw [line width=0.2pt,color=gray] (-5,14)-- (-5,18);
\draw [line width=2.pt](-8,0)-- (-8,4);
\draw [line width=2.pt] (-8,4)-- (-8,8);
\draw [line width=0.2pt,color=gray] (-8,8)-- (-8,12);
\draw [line width=2.pt] (-8,12)-- (-8,16);
\draw [line width=2.pt] (-3,1)-- (-3,5);
\draw [line width=2.pt] (-3,5)-- (-3,9);
\draw [line width=0.2pt,color=gray] (-3,9)-- (-3,13);
\draw [line width=2.pt] (-3,13)-- (-3,17);
\draw [line width=2.pt] (-4,4)-- (-4,8);
\draw [line width=0.2pt,color=gray] (-4,8)-- (-4,12);
\draw [line width=2.pt] (-4,12)-- (-4,16);
\draw [line width=2.pt] (-4,16)-- (-4,20);
\draw [line width=2.pt] (-8,0)-- (-7,2);
\draw [line width=0.2pt,color=gray] (-8,4)-- (-7,6);
\draw [line width=2.pt] (-8,8)-- (-7,10);
\draw [line width=2.pt] (-8,12)-- (-7,14);
\draw [line width=2.pt] (-8,16)-- (-7,18);
\draw [line width=0.2pt,color=gray] (-5,2)-- (-7,2);
\draw [line width=2.pt] (-5,6)-- (-7,6);
\draw [line width=0.2pt,color=gray] (-5,10)-- (-7,10);
\draw [line width=2.pt] (-5,14)-- (-7,14);
\draw [line width=0.2pt,color=gray] (-5,18)-- (-7,18);
\draw [line width=2.pt] (-5,2)-- (-4,4);
\draw [line width=2.pt] (-5,6)-- (-4,8);
\draw [line width=2.pt] (-5,10)-- (-4,12);
\draw [line width=0.2pt,color=gray] (-5,14)-- (-4,16);
\draw [line width=2.pt] (-5,18)-- (-4,20);
\draw [line width=2.pt] (-5,2)-- (-3,1);
\draw [line width=0.2pt,color=gray] (-5,6)-- (-3,5);
\draw [line width=2.pt] (-5,10)-- (-3,9);
\draw [line width=2.pt] (-5,14)-- (-3,13);
\draw [line width=2.pt] (-5,18)-- (-3,17);

\draw [line width=2.pt] (-7,2)-- (-9,3);
\draw [line width=2.pt] (-7,6)-- (-9,7);
\draw [line width=2.pt] (-7,10)-- (-9,11);
\draw [line width=0.2pt,color=gray] (-7,14)-- (-9,15);
\draw [line width=2.pt] (-7,18)-- (-9,19);
\draw [line width=2.pt] (-9,3)-- (-9,7);
\draw [line width=0.2pt,color=gray] (-9,7)-- (-9,11);
\draw [line width=2.pt] (-9,11)-- (-9,15);
\draw [line width=2.pt] (-9,15)-- (-9,19);

\begin{scriptsize}
\draw [fill=white] (-7,2) circle (2.5pt);
\draw [fill=white] (-5,2) circle (2.5pt);
\draw [fill=white] (-3,1) circle (2.5pt);
\draw [fill=white] (-4,4) circle (2.5pt);
\draw [fill=white] (-7,6) circle (2.5pt);
\draw [fill=white] (-5,6) circle (2.5pt);
\draw [fill=white] (-4,8) circle (2.5pt);
\draw [fill=white] (-3,5) circle (2.5pt);
\draw [fill=white] (-8,4) circle (2.5pt);
\draw [fill=white] (-7,10) circle (2.5pt);
\draw [fill=white] (-7,14) circle (2.5pt);
\draw [fill=white] (-8,8) circle (2.5pt);
\draw [fill=white] (-8,12) circle (2.5pt);
\draw [fill=white] (-5,10) circle (2.5pt);
\draw [fill=white] (-4,12) circle (2.5pt);
\draw [fill=white] (-3,9) circle (2.5pt);
\draw [fill=white] (-5,14) circle (2.5pt);
\draw [fill=white] (-4,16) circle (2.5pt);
\draw [fill=white] (-3,13) circle (2.5pt);
\draw [fill=white] (-8,0) circle (2.5pt);
\draw [fill=white] (-9,19) circle (2.5pt);
\draw [fill=white] (-9,15) circle (2.5pt);
\draw [fill=white] (-9,11) circle (2.5pt);
\draw [fill=white] (-9,7) circle (2.5pt);
\draw [fill=white] (-9,3) circle (2.5pt);
\draw [fill=white] (-7,18) circle (2.5pt);
\draw [fill=white] (-5,18) circle (2.5pt);
\draw [fill=white] (-3,17) circle (2.5pt);
\draw [fill=white] (-4,20) circle (2.5pt);
\draw [fill=white] (-8,16) circle (2.5pt);
\end{scriptsize}
\end{tikzpicture}
\bigskip

{\bf Figure 4. } A Hamiltonian cycle of $P_5\square G_2$
\end{center}
\bigskip

If $P_5\square G$ is Hamiltonian for any $G$ such that $G$ contains a path factor and $P_5\square G$ is balanced bipartite, then maybe it is possible to construct a Hamiltonian cycle for $P_{2k+5}\square G$ by combining a Hamiltonian cycle of $P_5\square G$ and a Hamiltonian cycle of $P_{2k}\square G$. Hence we give another conjecture.

\begin{conj}
Let $G$ be a graph with a path factor and $n\geq 4\Delta(G)-2$. If $P_n\square G$ is balanced bipartite, then $P_n\square G$ is Hamiltonian.
\end{conj}

\section*{Acknowledgments}

This research is supported by the Ministry of Science and Technology
of Taiwan under the project MOST 107-2115-M-009-009-MY2.

\end{document}